\documentclass[reqno]{amsart}
\usepackage{enumerate}
\usepackage{nicefrac}
\usepackage{amsmath} 
\usepackage{amssymb,amsthm}
\usepackage{microtype}
\usepackage{graphicx}
\usepackage{hyperref}
\usepackage{xcolor}
\hypersetup{colorlinks   = true,
     citecolor    = teal,
     urlcolor     = teal,
     linkcolor    = teal}

\newtheorem{thm}{Theorem}{\bf }{\it }
{\bf }{\it }
{\bf }{\it }
{\bf }{\it }
\newtheorem{conj}[thm]{Conjecture}{\bf }{\it }
\newtheorem{quest}[thm]{Question}{\bf }{\it }
{\bf }{\it }
{\bf }{\it }
{\bf }{\it }
\newtheorem{lem}[thm]{Lemma}{\bf }{\it }
{\bf }{\it }
\newtheorem{fct}[thm]{Fact}{\bf }{\it }
\theoremstyle{definition}
\newtheorem{defn}[thm]{Definition}{\bf }{\rm }
\newtheorem{rem}[thm]{Remark}{\bf }{\rm }
{\bf }{\rm }

\newtheoremstyle{TheoremNum}%
{}{}              
{\it}                      
{}                              
{\bf}                     
{.}                             
{ }                             
{\thmname{#1}\thmnote{ \bfseries #3} (restated)}
\theoremstyle{TheoremNum}
\newtheorem{thmn}{Theorem}

\def\wa{\widehat{a}}
\def\gcd{\mathrm{gcd}}
\def\rad{\mathrm{rad}}

\def\Z{\mathbb{Z}}
\def\mod#1{{\;(\textrm{mod}\;#1)}}

\def\br#1{\textcolor{blue}{#1}}
\def\rr#1{\textcolor{red}{#1}}
\def\br#1{#1}
\def\rr#1{#1}

\begin{document}
\title[Improved lower bounds for strong $n$-conjectures]{Improved lower bounds for\\strong $n$-conjectures}

\author{Rupert H\"olzl}
\address{Rupert~H\"olzl and Sören~Kleine, Fakult\"at f\"ur Informatik,
	Universit\"at der Bundeswehr M\"unchen, 
	Neubiberg, Germany}
\email{r@hoelzl.fr}
\email{soeren.kleine@unibw.de}

\author{Sören Kleine}
\address{Frank~Stephan, Department of Mathematics \& School of Computing,
	National University of Singapore, Singapore 119076,
	Republic of Singapore}
\email{fstephan@comp.nus.edu.sg}
\thanks{F.~Stephan's research was supported by the Singapore Ministry of Education AcRF Tier 2 grant MOE-000538-00 and AcRF Tier 1 grants A-0008454-00-00 and A-0008494-00-00.}

\author{Frank Stephan}

\begin{abstract}
The well-known $abc$-conjecture concerns triples $(a,b,c)$ of non-zero integers that 
are coprime and satisfy ${a+b+c=0}$. The strong $n$-conjecture is a generalisation to $n$ summands where integer solutions of the equation ${a_1 + \ldots + a_n = 0}$ are considered such that the $a_i$ are pairwise coprime and satisfy a certain subsum condition. Ramaekers studied a variant of this conjecture with a slightly different set of conditions. He conjectured that in this setting the limit superior of the so-called qualities of the admissible solutions equals~$1$ for any~$n$.

In this article, we follow results of Konyagin and Browkin. We restrict to a smaller, and thus more demanding, set of solutions, and improve the known lower bounds on the limit superior: for ${n \geq 6}$ we achieve a lower bound of $\nicefrac54$; for odd $n \geq 5$ we even achieve $\nicefrac53$. In particular, Ramaekers's conjecture is false for every ${n \ge 5}$.
\end{abstract}

\maketitle

\section{Introduction}

\noindent
The $abc$-conjecture \cite{Mas85,Oes88,Wal15} is a well-known open
problem in mathematics. It postulates that there is no constant
$q > 1$ such that there exist infinitely many triples $(a,b,c)$ of coprime
and nonzero integers with $a+b+c=0$ and such that the ``quality'' of $(a,b,c)$ exceeds $q$.

More precisely, the \emph{radical} $\rad(n)$ of a non-zero integer $n$ is defined as the largest square-free positive divisor
		of $n$. Now let $(a,b,c) \in \Z^3$ be such that ${a,b,c \ne 0}$. Then the \emph{quality} of $(a,b,c)$ is defined as 
		\[ q(a,b,c) = \frac{\log(\max\{|a|,|b|,|c|\})}{\log(\rad(a \cdot b \cdot c))}. \] 
For example, given the
triple ${(8192,-8181,-11) = (2^{13},-3^4 \cdot 101, -11)}$, its entries
are pairwise coprime, their largest square-free positive divisor is
$6666 = 2 \cdot 3 \cdot 11 \cdot 101$,
and its quality is $\log(8192)/\log(6666) \approx 1.0234$, seemingly supporting the claim of the $abc$-conjecture.

The conjecture itself is rather well studied
but still unresolved. However, on the way towards partial solutions, various
variants of the original problem were formulated and conjectures about the achievable qualities in these cases were made. While
Vojta~\cite{Voj87,Voj98} has studied a very general
statement that implies the $abc$\nobreakdash-conjecture,
a more immediate generalisation is the
$n$\nobreakdash-conjecture first studied by
Browkin and Brzezi\'nski~\cite{BB94}. 

The topic of this article is not this $n$-conjecture but two variants respectively introduced by Browkin~\cite{Bro00}, building on work of Konyagin, and Ramaekers~\cite{Ram09}; both used the term ``strong $n$-conjectures'' for their versions. Before we can state these conjectures, we first need to generalize the above definition of quality from triples to $n$-tuples.
\begin{defn}\label{quality}
For ${a = (a_1,\ldots,a_n) \in \Z^n}$, \rr{with $a_i\neq 0$ for $1 \leq i \leq n$,} we write 
\[q(a)  = 
\frac{\log(\max\{|a_1|,\ldots,|a_n|\})}{\log\,\rad(a_1 
\cdot \ldots \cdot a_n)}. \]
Then for a sequence $A = \{a^{(1)}, a^{(2)},\ldots\}  \subseteq \Z^n$ \rr{of $n$-tuples 
 as above,} let the {\em quality}
of $A$ be defined as
$$
   Q_A = \limsup_{k \rightarrow \infty} q(a^{(k)}). 
$$
\end{defn}

\noindent
Different strong $n$-conjectures concern the qualities of different 
sets $A$ of $n$-tuples of integers; it is not hard to see that $q$ and therefore $Q_A$ cannot take values lower than~$1$ for any~$A$.

\medskip
\noindent
Now we can state the strong $n$-conjectures mentioned above. We first recall the
$n$-conjecture and how it relates to the $abc$-conjecture.

\begin{conj}[$n$-conjecture; Browkin and Brzezi\'nski~\cite{BB94}] \label{conj:bb} $\,$ \\ 
	Let ${n \geq 3}$ and let ${A(n) \subseteq \Z^n}$ be the set of $n$-tuples 
	$(a_1,\ldots,a_n)$ such that \begin{enumerate}[(i)] 
		\item $a_1+\ldots+a_n=0$, 
		\item there are no ${b_1,\ldots,b_n \in \{0,1\}}$ and $i,j$ with 
		      ${1\leq i,j\leq n}$ such that ${b_i=0}$ and ${b_j=1}$ and 
		      ${\sum_{k=1}^n b_k \cdot a_k = 0}$, 
		\item $\gcd(a_1,\ldots,a_n) = 1$. 
	\end{enumerate} 
	Then $Q_{A(n)}=2n-5$ for every $n$.  
\end{conj}

\noindent
In the following we will informally refer to condition~(ii), as well as to analogous statements introduced below, as the \emph{subsum condition}. Note that in the case $n=3$ this condition excludes only finitely many triples and is therefore irrelevant for the value of $Q_{A(3)}$; this implies that the statement ``$Q_{A(3)}=1$'' is equivalent with the $abc$-conjecture. For larger~$n$, we have the following relationship:
\begin{thm}[Browkin and Brzezi\'nski~\cite{BB94}] \label{thm:abc-BB} 
	If the $abc$-conjecture is false then the $n$-conjecture is false
	for every $n \geq 4$.
\end{thm}

\noindent
One half of the $n$-conjecture is known: Browkin and Brzezi\'nski~\cite[Theorem~1]{BB94} proved for ${n \ge 3}$ that 
$Q_{A(n)} \ge 2n - 5$. This statement is not hard to prove; we will come back to it in Remark~\ref{rem:halb} at the end of this article.

Different conjectures arise when considering different sets $A$ and
one of the main goals of
this article is to clarify the relation between these different
conjectures and to try to unify the picture.

Browkin~\cite{Bro00} introduced the following conjecture he referred to as ``strong $n$\nobreakdash-con\-jec\-ture.''
It is obtained from the $n$-conjecture by requiring that the entries in each
$n$-tuple are pairwise coprime and removing the subsum condition. 
\begin{conj}[Browkin~\cite{Bro00}] \label{conj:bro} $\,$ \\ 
	Let $n \geq3$ and let $B(n)$ be the set of $n$-tuples 
	${(a_1,\ldots,a_n)\in \Z^n}$ such that \begin{enumerate}[(i)] 
	\item $a_1+\ldots+a_n=0$, 
	\item $\gcd(a_i,a_j) = 1$ for all $i,j$ with ${1 \leq i < j \leq n}$. 
	\end{enumerate} 
	Then $Q_{B(n)}<\infty$ for every $n$.
\end{conj}

\goodbreak 
\begin{rem}\quad \nopagebreak
	\begin{enumerate}[(1)] 
		\item The $abc$-conjecture corresponds to the statement ``$Q_{B(3)}=1$.''
		\item If $Q_{A(4)} \leq 3$, then ${Q_{B(3)} = 1}$. Indeed, 
		assume that there are infinitely many counterexamples $(a,b,c)$
to the $abc$-conjecture of quality at least $q$ with~${q>1}$. Then $Q_{A(4)} \geq 3q$ is witnessed by the quadruples
\[(a^3,b^3,c^3,-3abc).\] 
\item Similarly, ${Q_{A(5)} \leq 5}$ implies that ${Q_{B(3)} = 1}$ via quintuples of the form
\[(a^5,b^5,c^5,-5abc^3,5a^2b^2c).\]
\item More generally, if ${Q_{A(n)} \leq 2n-5}$ for some ${n \geq 4}$, since the reverse inequality is known as mentioned above, it would follow that the $n$-conjecture is true for this particular $n$. As a consequence, in view of Theorem~\ref{thm:abc-BB}, the $abc$-conjecture would be true as well in this case. 
\end{enumerate} 
\end{rem}

\noindent
Konyagin established the following result about Conjecture~\ref{conj:bro}.

\begin{thm}[Konyagin; see Browkin\footnote{We point out that there is
		a typo when Browkin states Konyagin's result; where we say
		``$n \geq 5$'' he says ``$n \geq 3$''. But then Theorem~\ref{konya} would
		already disprove the $abc$-conjecture. Indeed, Konyagin's proof only
		works for odd $n \geq 5$.}\!~\cite{Bro00}]\label{konya}
	\[Q_{B(n)} \geq 
	\begin{cases}
		1 & \text{if } n \geq 4 \text{ is even,}\\
		\nicefrac32 &  \text{if } n \geq 5 \text{ is odd.}\\
	\end{cases}\]
\end{thm}

\noindent
For completeness we mention that Konyagin's result can also be derived from an example given by Darmon and Granville~\cite[item (d) on page 542]{DG95} by choosing~${t = 2^k}$; they cite correspondence with Noam D.\ Elkies
as the source.

\medskip
\noindent
Another variant of the $n$-conjecture that we study is the following.

\begin{conj}[Ramaekers~\cite{Ram09}]\label{conj_rae} $\,$ \\ 
	Let $n \geq3$ and let $R(n)$ be the set of $n$-tuples ${(a_1,\ldots,a_n)\in \Z^n}$ such that 
	\begin{enumerate}[(i)] 
		\item $a_1+\ldots+a_n=0$, 
		\item there are no ${b_1,\ldots,b_n \in \{0,1\}}$ and $i,j$ with 
		${1\leq i,j\leq n}$ such that ${b_i=0}$ and ${b_j=1}$ and 
		${\sum_{k=1}^n b_k \cdot a_k = 0}$, 
		\item ${\gcd(a_i,a_j) = 1}$ for $i,j$ with ${1 \leq i < j \leq n}$. 
	\end{enumerate} 
	Then $Q_{R(n)}=1$ for every $n$. 
\end{conj}

\noindent
Note that Ramaekers' conjecture maintains the subsum condition
from the original $n$-conjecture, unlike Browkin's. 
Darmon and Granville~\cite[end of Section 5.2]{DG95} also mention this
statement as the ``generalised $abc$-conjecture,'' but only
conjecturing $Q_{R(n)} < \infty$ and without clarifying whether they
require pairwise or setwise coprimeness.

Except for $(1,-1,0)$ and its reorderings, all triples
in $B(3)$ are also in~$R(3)$, thus the $abc$-conjecture is 
equivalent to the claim that ${Q_{R(3)} = 1}$. Ramaekers computed numerous example elements of $R(3)$, $R(4)$ and $R(5)$ of quality larger than $1$. Here, the examples in~$R(4)$ exhibited a tendency of being of smaller quality than those in~$R(3)$, which could make one  suspect that disproving the claim ``$Q_{R(4)} = 1$'' might be even harder than disproving the $abc$-conjecture. We are however unaware of any known implications between the cases $n=3$ and $n=4$; for larger $n$, though, we will see below that~$Q_{R(n)} > 1$.

As $R(n)$~is a strictly smaller set than~$B(n)$, a priori $Q_{R(n)}$~could be smaller than~$Q_{B(n)}$. Thus, we cannot directly deduce anything about $Q_{R(n)}$ from Theorem~\ref{konya}; indeed, for odd $n \geq 7$,
Konyagin's proof of Theorem~\ref{konya} uses $n$-tuples which are
in~$B(n)\setminus R(n)$.

In this article we introduce two new restrictions, namely a \emph{stronger subsum condition} on the one hand, and the 
{\em set of forbidden factors} $F$ on the other hand. We will work with the following definition which is purposely designed for 
proving lower bounds on $Q_{R(n)}$; see Fact~\ref{compare_fact} below. 
\begin{defn}\label{def:U}  
	Let $n \geq 3$ and let ${F \subseteq \mathbb{N}}$ be a finite set, \rr{where ${\min F \geq 3}$ in case that $F\neq\emptyset$.}  
	We let $U(F,n)$ contain all
	${(a_1,\ldots,a_n)\in\Z^n}$ satisfying the following conditions:
	\begin{enumerate}[(i)]
		\item $\gcd(a_i,a_j) = 1$ for $i,j$ with $1 \leq i < j \leq n$;
		\item $a_1+\ldots+a_n=0$;
		\item there are no $b_1,\ldots,b_n \in \{-1,0,1\}$ and $i,j$ with
		$1\leq i,j\leq n$ such that $b_i=0$ and $b_j=1$ and
		$\sum_{k=1}^n b_k \cdot a_k = 0$;
		\item none of the numbers $a_1,\ldots,a_n$ is a multiple of any number
		in $F$. 
	\end{enumerate}
\end{defn}

\begin{rem}\quad \nopagebreak
	\begin{enumerate}[(1)]
		\item If $F$~is empty then condition~(iv) is vacuously satisfied by every $n$-tuple. 
		
		\item 
		If~${2 \in F}$
		and $n$~is odd, then ${U(F,n) = \emptyset}$ since the sum of an odd number
		of odd integers cannot be~$0$. For the case that $n$~is even, note that by condition~(i) at most one of the $a_i$'s can be even; but then by condition~(ii) no $a_i$ can be even. Thus the assumption $2\in F$ is unnecessary in this case, and can be omitted. In summary, we do not consider the case $2\in F$.
	\end{enumerate}
\end{rem}

\noindent We are interested in questions of the following type.

\begin{quest} 
	Fixing different choices of $F$ and $n$,
	what are valid lower bounds on $Q_{U(F,n)}$?
\end{quest}

\noindent
While Browkin and Brzezi\'nski opted to only allow coefficients $b_j\in\{0,1\}$ in the subsum condition in Conjecture~\ref{conj:bb}, our new condition (iii) above is more demanding as it allows negative coefficients as well. Thus the quality lower bounds we will establish below are proven for a smaller set of $n$-tuples and will therefore also hold for the conjectures stated above. More precisely stated,
the following relationships between the different strong $n$-conjectures are immediate.

\begin{fct}\label{compare_fact}
	For every ${n \in \mathbb{N}}$ and any $F$ as above we have ${Q_{U(F, n)} \leq Q_{A(n)}}$ as well as ${Q_{U(F, n)} \leq Q_{R(n)} \leq Q_{B(n)}}$.
\end{fct}

\noindent
This means in particular that, by fixing the right parameters, our new definition provides a framework which can be used to prove lower bounds on both Browkin's and Ramaekers' versions of the problem. 

\medskip
\noindent
In the remainder of this article, we will prove lower bounds for $Q_{U(F,n)}$ for suitable parameters~$F$ and~$n$. First, we will improve Konyagin's construction cited above to obtain
the following stronger version of his result. 
\begin{thm}\label{th:first}
	Let $F$ be such that 
	$2,5,10 \notin F$. Then ${Q_{U(F,n)} \geq \nicefrac53}$ for each odd ${n \geq 5}$. In particular, ${Q_{U(\emptyset,n)} \geq \nicefrac53}$ for these $n$. 
\end{thm}

\noindent
In particular, Ramaekers' conjecture is wrong for odd ${n \geq 5}$. Even integers are covered by our second main result, which holds for arbitrary ${n \ge 6}$ and arbitrary finite~$F$.

\begin{thm} \label{th:third}
	Let $n \geq 6$ and let $F$ be an arbitrary finite set. Then
	\[Q_{U(F,n)} \geq \nicefrac54.\]
	In particular, $Q_{R(n)} \geq \nicefrac54$ for each ${n \geq 6}$. 
\end{thm}

\noindent
We stress that these results disprove Ramaekers' conjecture for any ${n \geq 5}$. 

Finally, we will conclude with a brief discussion of $n$-tuples $(a_1, \ldots, a_n)$ that are coprime but not necessarily 
pairwise coprime, with a particular focus on Conjecture~\ref{conj:bb} of Browkin and Brzezi\'nski.

\section{The case of odd $n \geq 5$}

\noindent
As a warm-up and an illustration of Konyagin's technique, we first give a proof of a weaker version of Theorem~\ref{th:first} for~$n=5$. In the process we slightly modify the construction that he used to prove Theorem~\ref{th:konyagin}, so as to obtain a bound on $Q_{U(\emptyset,5)}$ in place of $Q_{B(5)}$.

\begin{thm}\label{th:konyagin} 
	${Q_{U(\emptyset,5)} \geq \nicefrac32}$. 
\end{thm}

\begin{proof}
Fix any integer~$k \geq 1$ and let 
\[{a = (6^{2^k}+1)^3},\quad {b =
-(6^{2^k}-1)^3},\quad {c = -6 \cdot (6^{2^k})^2}, \quad d = -31, \quad e = 29.\]
Then, on the one hand, $\log(a) \geq 3 \cdot 2^k \cdot \log(6)$ holds; and, on the other hand, $\rad(a \cdot b \cdot c \cdot d \cdot e)$ is a factor of
${(6^{2^k}+1) \cdot (6^{2^k}-1) \cdot 6 \cdot 31 \cdot 29}$, so that  its
logarithm must be bounded by $2 \cdot 2^k \cdot \log(6) + \ell$ for some constant~$\ell$. Thus,
\[
   q(a,b,c,d,e) \geq 
   \frac{3 \cdot 2^k \cdot \log(6)}{2 \cdot 2^k \cdot \log(6) + \ell},
\]
which converges to $\nicefrac32$ for $k \rightarrow \infty$.

\medskip
\noindent
We claim that for every $k\geq 1$, if $a,b,c,d,e$ are chosen as above, then 
they are pairwise coprime. If we write $s = 6^{2^k}$, then $a$, $b$ and $c$ are of the forms $(s+1)^3$, $-(s-1)^3$ and $-6s^2$,
respectively. Trivially, $s-1$ and~$s$ are coprime, and the same holds for $s$ and $s+1$. As $2$ and $3$ are
the only factors of $s$, neither of them can be a factor of $s-1$ or
$s+1$, and thus $(s+1)^3$ and $6s^2$, as well as $(s-1)^3$ and $6s^2$,
are coprime. As $s-1$ and~$s+1$ are both odd, they cannot have $2$ as a common factor, and thus $s-1$ and $s+1$ must be coprime; consequently, $(s-1)^3$ and $(s+1)^3$ are coprime. To complete the argument, consider the sequence $(6^{2^k})_{k \geq 1}$; if we can show that, modulo $29$ and modulo $31$, none of its elements
equals $-1$, $0$, or $1$, then none of $s-1$, $s$ or $s+1$ can be a 
multiple of $29$ or $31$, implying that each of $a,b,c$ is coprime
with both $d=29$ and $e=31$. We proceed by 
repeated squaring; first we obtain  
\[\begin{array}{l@{\;}c@{\;}c@{\;}c@{\;}r@{\;}c@{\;}r@{\;}l}
6&&&&&\equiv &	6    & \pmod{29}   \\
	6^2    &&& =      & 36     & \equiv & 7                & \pmod{29} \\
	6^4    & \equiv & 7^2    & =      & 49 & \equiv & -9           & \pmod{29} \\
	6^8    & \equiv & (-9)^2 & =      & 81 & \equiv & -6          & \pmod{29} \\
	6^{16} & \equiv & (-6)^2 & =      & 36 & \equiv & 7 &\pmod{29}, 
\end{array}\]
and so on. Similarly, modulo $31$, we obtain the sequence $6$, $5$,
$-6$, $5$, and so on. Thus, $a,b,c,d,e$ are pairwise coprime,
establishing condition~(i) of Definition~\ref{def:U}.

Condition~(ii) is immediate.
For condition~(iii), assume that there exist non-trivial subsums equaling~$0$ and fix one. Clearly, no combination of only the elements $c$,~$d$ and~$e$ exists that sums to~$0$. Thus at least one of $a$ or $b$ must occur in our subsum. But if
${\pm (s+1)^3}$ is part of the subsum, so must ${\mp (s-1)^3}$, as otherwise there would be no hope of the subsum equaling $0$. Also, the signs of these two numbers must clearly be opposite; w.l.o.g.~assume that they are chosen in such a way that the sum of these two elements is positive, that is, that it equals $6 \cdot s^2+2$. Then $-6s^2$ must also be part of the subsum in order to have any hope of achieving a subsum equaling $0$. But 
$ {(s+1)^3 - (s-1)^3 -6s^2 = 2}$, and thus the only way to achieve a sum of $0$ in this case is by also including $29$ and $-31$. Thus all five of  $a,b,c,d,e$ are required in a subsum for it to equal~$0$; this contradicts our assumption that our subsum was a non-trivial example.

Finally, condition~(iv) of Definition~\ref{def:U} is vacuous as ${F = \emptyset}$. 
\end{proof}

\noindent
To obtain the stronger Theorem~\ref{th:first} stated in the introduction, we will use a proof that is similar to the previous one, except that we employ a degree $5$ polynomial instead of a degree $3$ one to obtain a better
bound. We begin by proving an auxiliary result. 

\begin{lem} \label{pr:factoravoidance}
	Let \br{$u,m \in \mathbb Z$ with ${u < 0 < m}$ and ${m \geq \max(2,|u|)}$,} write 
	\[ q=\prod_{p \leq
		m \,\wedge\, p \text{ prime}} \, p. \] 
		and let ${F = \{3,4,\ldots,m\}}$.
Then there are a natural number ${v > 0}$ and an odd integer~${w \leq 0}$ 
	with ${u=v+w}$ such that
	\begin{itemize}
		\item ${q < v \le |w| \leq (m+1) \cdot q}$,
		\item ${\gcd(v,w)=1}$, and
		\item no element of
		$F$ divides $v$ or $w$.
	\end{itemize}
\end{lem}

\begin{proof}
	Let $q$ be as in the statement. 
	We run the following algorithm:
	
	\smallskip
	
	\begin{samepage}
		\begin{enumerate}[\qquad(1.)~\mbox{ }]
			\item Let $v = u+1+q$ and $w=-q-1$.
			\item For all prime numbers $3 \leq p \leq m$,
			\item \quad
			while $p$ divides one of $v$ or $w$, 
			\item \quad \quad 
			let $v = v+q/p$ and $w = w-q/p$.
			\item If $4$ divides $v$ then let $v = v+q$ and $w = w-q$.
		\end{enumerate}
	\end{samepage}
	
	\medskip
	\noindent
	Note that the sum $v+w=u$ and the fact that $w$ is odd are invariants
	during the execution of this algorithm. Further note that ${q < v}$ and ${|w| \leq (m+1) \cdot q}$ are immediate by construction.
	
	During the ``for'' loop over $p$, since $q/p$ is not a multiple of
	$p$, only one of~$v$, $v+q/p$, or $v+q/p+q/p$ 
	can be a multiple of $p$. The same applies to~$w$,
	$w-q/p$, and $w-q/p-q/p$. Thus, for each $p$, the instruction inside
	the ``while'' loop will be executed $0$, $1$ or~$2$~times, and
	afterwards neither $v$ nor $w$ will be divisible by~$p$.
	
	We claim that, once established, this property is preserved throughout
	the rest of the algorithm: Consider some prime $p' \neq p$
	which was handled in a previous iteration of the ``while'' loop, and
	assume that at the beginning of the iteration for~$p$ 
	we have that neither $v$ nor $w$ are divisible by $p'$.
	Since $q/p$ is a multiple of~$p'$, we have $v \equiv v + q/p \mod{p'}$
	and $w \equiv w - q/p \mod{p'}$; thus the property is preserved by the action taken at line~(4.). For similar reasons, the property
	also is preserved during the final execution of~(5.). This
	proves the claim, and it follows that after the algorithm terminates,
	$v$ and $w$ are not divisible by any odd prime less than or equal to~$m$.
	
	Assume that  $v$ is divisible by $4$ before the execution of line~(5.).
	Then, since $q$ is
	not divisible by $4$, $v+q$ is an even number {\em not} divisible
	by~$4$. Thus, in any case, after the execution of line~(5.), $v$ is
	not divisible by~$4$. Since $w$ was odd, it is still odd after the
	execution of line~(5.); in particular it is not divisible by~$4$.
	
	Overall we have established that, when the algorithm terminates, none of
	the numbers $3,4,\ldots,m$ divide $v$ or $w$.
	
	\medskip
	\noindent
	To see that $v$ and $w$ are coprime, first note that $2$ cannot be a
	common prime factor since $w$~is odd. By construction, any odd common
	prime factor $p$ of $v$ and $w$ must be larger than $m$. But any such
	$p$ also is a prime factor of $u=v+w$, which is impossible as $u \leq m$.
	
	\br{Finally, since ${v + w = u}$ and ${u < 0}$ it is obvious that ${v \le |w|}$.} 
\end{proof}

\noindent
With this established, we are ready to prove the first main result of this article. We point out that it is closely related to an observation of Ramaekers~\cite[Section 4.4]{Ram09}; he gives credit for the idea to use polynomial identities to the previously mentioned examples of
Darmon and Granville~\cite{DG95} and Elkies.
For these, the condition that the~$a_i$'s have to be pairwise coprime is dropped; see~Remark~\ref{rem:halb}. 

\begin{thmn}[\ref*{th:first}]
	Let $F$ be such that 
	$2,5,10 \notin F$. Then ${Q_{U(F,n)} \geq \nicefrac53}$ for each odd ${n \geq 5}$. In particular, ${Q_{U(\emptyset,n)} \geq \nicefrac53}$ for these $n$. 
\end{thmn}
\begin{proof}
We will construct infinitely many $n$-tuples $(a_1,\ldots,a_n)$ where 
\begin{itemize}
	\item $a_1=(x-1)^5$, 
	\item $a_2=10(x^2+1)^2$,
	\item $a_3=-(x+1)^5$. 
\end{itemize}
We will then show that there exist choices for $a_4,\ldots,a_n$ that only depend on $n$ and such that there are infinitely many~$x$ such that these $n$-tuples satisfy the conditions posited by Definition~\ref{def:U}. We begin by letting 
\[\wa_4 = \begin{cases}
	24 & \text{if } F = \emptyset,\\
	3 \cdot (8+\max(F)) & \text{otherwise.}
\end{cases}\]
For ${i = 4,5,\ldots,n-2}$ we proceed inductively by letting each
$a_i$ be a prime number larger than $\wa_i$ and by letting each ${\wa_{i+1} = 3 \cdot a_i}$. 

Next, we let $a_{n-1}$ and $a_n$ be the numbers $v$ and $w$ provided by Lemma~\ref{pr:factoravoidance} when applied with parameters
\begin{itemize}
	\item ${u = -(8+a_4+a_5+\ldots+a_{n-2})}$
	\item ${m = \wa_{n-2}}$;
\end{itemize}
in particular  ${a_{n-1} > 0}$. Finally, let ${\wa_n = 3 \cdot (|a_{n-1}|+|a_n|)}$.

\br{Note that
${(x-1)^5 + 10(x^2+1)^2 -(x+1)^5 = 8}$ holds independently of the
choice of~$x$; thus by choice of $u$ we have ${a_1+a_2+\ldots+a_n = 0}$. Recall that $n$ is odd by assumption. As~${a_4+a_5+\ldots+a_{n-2}}$ is composed of an even number of all odd summands, $u$ must be even.} Therefore $a_{n-1}$~and~$a_n$ must have the same parity; however, by Lemma~\ref{pr:factoravoidance} they cannot both be even. Thus it follows that all of ${a_4,a_5,\ldots,a_n}$ are odd; moreover, they are pairwise coprime by construction. 

\medskip
\noindent
Set ${y = \wa_n!}$ and consider the equation
\begin{align} \label{eq:Pell} 
y^2 \cdot s^2 - (y^2+1) \cdot t^2 = -1. 
\end{align} 
As there is an initial solution ${(s, t) = (1,1)}$ and as ${y^2 \cdot (y^2+1)}$ is positive and not a square, it follows that equation \eqref{eq:Pell}~has infinitely many integer solutions (see, for instance, Bundschuh~\cite[Subsection~4.3.7, page~198]{bundschuh}). 
Fix any solution $(s,t)$ of equation~\eqref{eq:Pell} and let ${x = y \cdot s}$.

Thus $x$ is a multiple of each element of $F$ and of each of~${a_4,a_5,\ldots,a_n}$; and therefore ${x-1}$, ${x+1}$ and ${x^2+1}$ are each coprime with any of these numbers. Furthermore, each of 
$2$, $5$, and $10$ is coprime with each of ${a_4,a_5,\ldots,a_n}$;
as a result ${10 (x^2+1)^2}$ is coprime with these numbers as well.
As $x$ is even, ${(x-1)^5}$ and~${(x+1)^5}$ are coprime and
${(x^2+1)}$ is coprime with ${x^2-1}$ and thus with ${x-1}$ and ${x+1}$ as well. As $10$ divides $x$, the numbers ${x-1}$, ${x+1}$ and ${x^2+1}$ are coprime with $10$. Also, no element of $F$ divides any of ${a_1,\ldots,a_n}$. 

In summary, conditions~(i),~(ii), and~(iv) in Definition~\ref{def:U} are satisfied. Now assume there exists a non-trivial zero subsum, that is, that there are ${b_1, \ldots, b_n}$ such that ${b_1 \cdot a_1 + \ldots + b_n \cdot a_n}=0$ and such that not all $b_i$'s equal $0$. We distinguish two cases:

If $b_1,b_2,b_3$ are not all equal, then $(b_1, b_2, b_3)$ or $(-b_1, -b_2, -b_3)$ must equal one of 
\[
\begin{array}{ccccc}
	(1,0,0), & (0,1,0),  & (0,0,1),  & (1,0,1),  & (1,0,-1),  \\
	(1,1,0), & (1,-1,0), & (1,-1,1), & (1,1,-1), & (1,-1,-1).
\end{array} 
\]
Recalling that $x$ is a multiple of $y$, it is easy to verify that in each of these cases we have  
$|b_1 \cdot a_1 + b_2 \cdot a_2 + b_3 \cdot a_3|>y$. But then, since their absolute values are too small compared with~${y = \wa_n!}$, no combination of the remaining~$a_i$ with ${i \geq 4}$ is possible that would lead to a zero subsum.

In the other case, if $b_1=b_2=b_3$, then their sum is either $-8$ or $0$ or $+8$. We distinguish all three possible cases concerning the value of $b_n$:
\begin{itemize}
	\item {\em If $b_n=0$, then the subsum is empty:}
	This is because in the sequence
	\[|a_1+a_2+a_3|,|a_4|,|a_5|,\ldots,|a_{n-1}|\]
	each entry is at least $3$ times larger than the previous one; thus the only way of obtaining a zero subsum in this case is when ${b_k=0}$ for all ${1\leq k \leq n}$.

	\item {\em If $b_n=1$, then $b_k=1$ for all ${1\leq k \leq n}$:}
	Assume that for some choice of $(b_k)_{1\leq k \leq n}$ with ${b_n=1}$ we have ${\sum_{k=1}^n b_k \cdot a_k = 0}$. Since we also have ${\sum_{k=1}^n a_k = 0}$ it follows that
	\[\begin{array}{c@{\;}l}
		&\sum_{k=1}^n a_k - \sum_{k=1}^n b_k \cdot a_k \\[0.5em]
		=& (1-b_1) \cdot (a_1+a_2+a_3)+\sum_{k=4}^{n-1} (1-b_k) \cdot a_k\\[0.5em]
		=& 0,
	\end{array}
	\]
where $1-b_k \in \{0,1,2\}$ for $k \in \{1,4,5,\ldots, n-1\}$. For the same reason as in the previous item, the only choice of $(1-b_k)_{k \in \{1,4,5,\ldots, n-1\}}$ that makes this equality true is $1-b_k =0 $ (thus $b_k =1 $) for all $k \in \{1,4,5,\ldots, n-1\}$.

	\item {\em If $b_n=-1$, then $b_k=-1$ for all $1\leq k \leq n$,} by a symmetric argument.
\end{itemize}

\noindent
In summary, the subsum condition~(iii) in Definition~\ref{def:U} is satisfied as well. 

It remains to estimate the qualities of the \rr{constructed} $n$-tuples. Note that the terms~$y$ and~$y^2+1$ as well as the terms~$a_4,\ldots,a_n$ are constant, and that by equation~\eqref{eq:Pell} the term ${a_2 = 10 (x^2+1)^2 = 10 (y^2+1)^2 \cdot t^4}$ only contributes a factor $t \in O(x)$ to the radical. Thus we have
$\rad(a_1 \cdot \ldots \cdot a_n) \in  {O(x^3)}$.

On the other hand, $\max\{|a_1|,\ldots,|a_n|\} = |a_3| = |x+1|^5$, and thus there is a constant $C$ such that we have
\[
q(a_1,\ldots,a_n) \geq 
\frac{\log(x+1)^5}{\log(Cx^3)}, 
\]
and therefore
\[
Q_{U(F,n)} \geq \lim_{x \rightarrow \infty} q(a_1,\ldots,a_n) \geq \lim_{x \rightarrow \infty} \frac{\log(x^5)}{\log(x^3) + \log(C)} = \nicefrac53.
\]
This completes the proof.
\end{proof}

\noindent
The above result only holds for $F$ not containing $2$, $5$ or $10$. In case we do allow $5$~or~$10$ in $F$, we can still obtain the following weaker lower bound by considering polynomials whose degree depends on $F$.

\begin{thm} \label{satz_polynom_begr} 
	Let $F$ be a finite set with $\min(F) \geq 3$. Then
	${Q_{U(F,5)} > 1}$.
\end{thm}

\begin{proof}
	As before, we may assume ${F=\{3,4,\ldots,m\}}$ for some~$m$. 
	Let ${s = h!-1}$ for ${h > 9m}$ and keep $h$ and $s$ constant during the
	remainder of the construction. Let~${x = k!}$ for some ${k>s}$;
	as in the previous construction, we will demonstrate that
	for sufficiently large~$k$ all required properties are satisfied.
	Then by letting $k$ go to infinity we obtain infinitely many examples
	that together witness the desired lower bound for $Q_{U(F,5)}$.
	
	We consider the following numbers; here, the choice of $a_1$, $a_2$, $a_3$ and $a_4+a_5$ follows Ramaekers~\cite[Section~4.4]{Ram09} but we then  additionally split $a_4+a_5$ into two summands:
	\begin{itemize}
		\item $a_1 = (x+1)^s$
		\item $a_2 = -(x-1)^s$
		\item $a_3 = -2s \cdot (x^2+(s-2)/3)^{(s-1)/2}$
		\item $a_4 = -(a_1+a_2+a_3+y)$ for some fixed odd ${y>s}$ that we choose below.
		\item $a_5 = y$
	\end{itemize}
	Note that, as a polynomial in $x$, we have that  $a_1+a_2$ is of degree $s-1$ and
	\textit{even}, that is, of the form $c_0 + c_2x^2+ c_4x^4 + c_6x^6+\ldots$\,. Similarly note that $a_1+a_2+a_3$ is an even polynomial in $x$ of degree~$s-5$.
	Finally note that, when dividing an even polynomial by a 
	polynomial of the form $x^2+c$, for some~$c \in \mathbb Z$, then the
	remainder is an integer not depending on $x$; if we write
	$z_0$, $z_1$ and $z_2$ for the remainders of $a_1+a_2+a_3$  modulo $x^2$, modulo $x^2-1$ and modulo $x^2+(s-2)/3$, respectively, 
	then the following auxiliary statement holds.
	\begin{lem} \label{lem:y} 
		We have that $6$ divides $z_0$ and that there exists an integer~$y$ such that 
		\begin{itemize}
			\item none of $y$, $y+z_0$, $y+z_1$, or $y+z_2$ has a prime factor~$q$
			where $$5 \leq q \leq (2s)^s+|z_0|+|z_1|+|z_2|,$$  
			\item neither 
			$y$ nor $y+z_0$ are divisible by $2$ or~$3$.
		\end{itemize}
	\end{lem} 
	\begin{proof}
		We achieve this by a similar method as in the
		proof of Lemma~\ref{pr:factoravoidance}:
		
		Let
		${b = (2s)^s+|z_0|+|z_1|+|z_2|}$, ${r = \prod_{q \leq b \,\wedge\, q \text{
					prime}}}$ and proceed as follows.
		
		\smallskip
		
		\begin{samepage}
			\begin{enumerate}[\qquad(1.)~\mbox{ }]
				\item Let $y=1$.
				\item For all primes $q$ with ${5 \leq q \leq b}$, 
				\item \quad replace $y$ by $\min(M \cap N)$ where 
				\[\begin{array}{l@{\;}c@{\;}l}
					\quad\quad M & = & \{y+i\cdot \nicefrac{r}{q}\colon\; 0 \leq i
					\leq 4\},                                             \\
					\quad\quad N & = & \{y'\colon\;  q \nmid y' \wedge q \nmid (y'+z_0) \wedge q
					\nmid (y'+z_1) \wedge q\nmid (y'+z_2)\}.
				\end{array}\] 
			\end{enumerate}
		\end{samepage}
		
		\smallskip
		\noindent Note that $q$ does not divide $\nicefrac{r}{q}$, and thus, for each 
		$$z \in \{y',y'+z_0,y'+z_1,y'+z_2\},$$ 
		at most one among 
		${z,z+\nicefrac{r}{q},\ldots,z+4 \cdot \nicefrac{r}{q}}$ can be a
		multiple of $q$. Thus, by the pigeonhole
		principle, the choice of $y$ in (3.) is always possible.
		
		\medskip
		\noindent
		That the final $y$ emerging from this process has the first of the two
		stipulated properties then follows from an argument analogous to that
		used in the proof of Lemma~\ref{pr:factoravoidance}.
		
		\smallskip
		\noindent
		To argue that $y$ and ${y + z_0}$ have the second property, we first prove that $z_0$
		is divisible by $6$. An 
		easy calculation shows that ${z_0=2-2s \cdot ((s-2)/3)^{(s-1)/2}}$, an
		even number. To see that ${z_0 \equiv 0 \mod 3}$, it is enough to show
		that 
		\[ 2s \cdot ((s-2)/3)^{(s-1)/2} \equiv 2 \mod 3.\] 
		To that end, note
		that, as~${h! \equiv 0 \mod 4}$, we have that ${s-1=h!-2 \equiv 2 \mod
			4}$, and thus that $(s-1)/2$ is odd. 
		Recall that ${s = h!-1}$, thus ${s \equiv 8 \mod 9}$. Now ${s-2 \equiv 6 \mod 9}$
		and ${(s-2)/3  \equiv 2 \mod 3}$. Moreover, ${2s \equiv 2 \cdot 2 \equiv 1 \mod 3}$. Therefore 
		$2s \cdot ((s-2)/3)^{(s-1)/2} \equiv 2 \mod 3$ and thus ${z_0 \equiv 0 \mod 6}$.
		
		\smallskip
		\noindent
		To complete the proof of the lemma, note that after the execution of~(1.) we have ${y + z_0 \equiv y \equiv  1 \mod 6}$. As all terms~$\nicefrac{r}{q}$
		appearing in the algorithm are multiples of $6$, this last property is
		invariant during the algorithm's execution, and the final $y$ and
		${y+z_0}$ are not divisible by $2$ or $3$.\phantom{\qedhere}\hfill$\Diamond$
	\end{proof} 

\noindent
	To continue with the proof of Theorem~\ref{satz_polynom_begr}, fix an integer $y$ as provided by Lemma~\ref{lem:y}; note that $y$ does not depend on $x$, a fact which will prove crucial in our closing arguments below. We verify conditions~(i)--(iv)
	stipulated by Definition~\ref{def:U}. Condition~(ii) is trivially satisfied by choice.
	
	By construction, $x$ is a multiple of $3$ while neither $s$ nor $(s-2)/3$ are multiples of~$3$ by the arguments given in the proof of Lemma~\ref{lem:y}; thus, $3$ does not divide $a_3=-2s \cdot (x^2+(s-2)/3)^{(s-1)/2}$. We further claim that $a_3$ is not divisible by~$4$ either; this is because $x$ is even, $s$ is odd, and $(s-2)/3$ is easily seen to be odd by construction. Now let ${q > 3}$ be a prime factor of any element of~$F$. By construction, $q$~divides~$x$ but neither $s$ nor ${(s-2)/3}$. It follows that none of ${x+1}$, ${x-1}$ and ${x^2-(s-2)/3}$ are  multiples of $q$. By Lemma~\ref{lem:y} neither $y$ nor ${y+z_0}$ are  divisible by $q$. Thus none of ${a_1,a_2,a_3,a_4,a_5}$ is a multiple
	of any element of $F$ and thus condition (iv) is satisfied. 
	
	Clearly, the fact that  $x$ is even implies that $a_1$ and $a_2$ are coprime by construction. Observe that ${(x^2+(s-2)/3) - (x^2-1) = \br{(s+1)/3 = h!/3}}$; this implies that if $x+1$ or $x-1$ have a common factor~$q$ with $a_3$, then $q$ must either divide~$2s$ or~\br{${(s+1)/3}$}. By construction, any such~$q$ also divides~$x$, which implies $q=1$. In summary we have that $a_1$, $a_2$ and $a_3$ are pairwise coprime.
	
	For sufficiently large~$k$ we have 
	\begin{align} \label{eq:yz1} k \geq 2s+|y|+|z_0+y|+|z_1+y|+|z_2+y|; 
	\end{align} 
	from now we assume that such a $k$ was chosen. Then a prime factor~$q$ of any of the summands in this inequality is also a factor of~${x = k!}$, and therefore not of $x-1$ or $x+1$. By Lemma~\ref{lem:y}, no prime
	factor $q$ of $y$, ${z_0+y}$, ${z_1+y}$, or ${z_2+y}$
	divides $2s$ or $(s-2)/3$ either. Altogether we obtain that no such~$q$  is a factor of $a_1$, $a_2$ or $a_3$, and therefore all  three must be coprime with~$a_5$. 
	
Next suppose that there exists a prime $q$ dividing both $a_3$ and $a_4$. As $a_4$ is odd, this would mean that 
	either $q$ divides $s$ or $q$ divides ${x^2+(s-2)/3}$. In the first case, $q$ would divide ${x = k!}$ since ${k > s}$. Therefore ${a_1 + a_2 + a_3 \equiv z_0 \mod{q}}$ and thus $a_4$ was congruent to ${-(z_0 + y)}$ modulo $q$. Since $q$ divides $a_4$ by assumption (and as we have already seen that $a_3$ is not divisible by 3), this would contradict the choice of~$y$ in Lemma~\ref{lem:y}. So suppose that $q$ divides ${x^2+(s-2)/3}$. Since ${a_4 \equiv -(z_2+y) \pmod{x^2+(s-2)/3}}$ 
	it would follow that $q$ is a prime factor of~${z_2 + y}$. In view of~\eqref{eq:yz1} this would imply that $q$~divides~${x = k!}$. But then $q$ would also divide $(s-2)/3$, which together with the fact that $q$ divides ${y + z_2}$ would again imply ${q \le 3}$. Since $q$ divides the odd $a_4$ and also $a_3$, which is not divisible by~$3$, this is impossible. 
	
	If a prime $q$ divides one of $a_1$ or $a_2$ then $q$ must also divide ${x^2-1}$. However, 
	\[a_4 \equiv -(z_1+y) \pmod{x^2-1};\]
	thus if $q$ would divide $a_4$ then it would also divide~${z_1+y}$. For $k$ large enough so that~\eqref{eq:yz1} holds, it would follow that $q$ divides~${x = k!}$, yielding a contradiction. 
	
Finally, we prove that $a_4$ and $a_5$ are coprime. First note that by~\eqref{eq:yz1} every prime factor of $a_5$ is a factor of $x$
	and, as ${a_4 \equiv 2s \cdot ((s-2)/3)^{(s-1)/2} \mod x}$,
	any common prime factor of $a_4$ and $a_5$ must be a factor of ${2s \cdot ((s-2)/3)^{(s-1)/2}}$. 
	But as we argued above, no prime factor of ${y = a_5}$ divides ${2s}$ or~${(s-2)/3}$. Thus $a_4$ and $a_5$ are coprime. In summary, condition (i) is satisfied.
	
	To see that condition (iii) is satisfied for all
	sufficiently large $k$, consider $a_1$, $a_2$, $a_3$ and~$a_4$ as polynomials in $x= k!$. In order for a subset of these numbers \br{or their negations} to sum to~$0$ all terms depending on $x$ need to be eliminated. To achieve this, if one of $a_1$ or $a_2$ is present in a subsum, \br{that is, if its coefficient is in $\{-1,1\}$,} the other clearly needs to be present \br{using the same coefficient} as well. First assume that they are both present; then their sum is of degree ${s-1}$; thus $a_3$ would be needed in the subsum as well \br{with a suitable coefficient taken from~$\{-1,1\}$}. \br{Regardless of the choice of coefficients,} the polynomials $a_1$, $a_2$ and $a_3$ cannot be combined in such a way as to produce a polynomial that is of degree less than $s-5$; which implies that we also need $a_4$. Finally, as $a_1+a_2+a_3+a_4=-y$ by definition, we also require $a_5$ in the subsum to make it equal~$0$. A similar argument applies if neither $a_1$ or $a_2$ are present in a subsum. We conclude that no non-trivial subsum can equal~$0$.

	\medskip
	\noindent
	We complete the proof by estimating the quality of~$(a_1,\ldots,a_5)$. We have that 
	\[\rad(a_1 \cdot \ldots \cdot a_5) \in  y \cdot O\left((x^2-1) \cdot (x^2+(s-2)/3) \cdot x^{s-5}\right)\!;\] 
	that is, using that $y$ is independent of~$x$, there exists a polynomial in $x$ of degree~${s-1}$ upper-bounding $\rad(a_1 \cdot \ldots \cdot a_5)$.
	
	Thus there is a constant $C$ such that for large enough $k$ we have
	\[
	q(a_1,\ldots,a_5) \geq 
	\frac{s \cdot \log(x+1)}{\log((x^2-1) \cdot (x^2+(s-2)/3) \cdot C x^{s-5} \cdot y)}, 
	\]
	and therefore, recalling that ${x = k!}$, 
	\[
	\lim_{k \rightarrow \infty} q(a_1,\ldots,a_5) \geq \lim_{k \rightarrow \infty} \frac{s \cdot \log(x)}{\log(x^4 \cdot x^{s-5} \cdot C')} = \frac{s}{s-1} > 1
	\]
	for some constant $C'$. We conclude that ${Q_{U(F,5)} > 1}$. 
\end{proof}

\noindent
Note that the value of $s$ in the proof depends on ${m = \max(F)}$ and therefore we cannot provide a fixed lower bound ${q > 1}$ working for any set $F$.

\section{The case of arbitrary ${n \ge 6}$} 

\noindent
The results obtained in the previous section concerned only odd~$n \geq 5$.
Here, we prove our next main result which holds true for general $n \geq 6$ and in particular refutes Ramaekers' conjecture for these~$n$.

\begin{thmn}[\ref*{th:third}] 
Let $n \geq 6$ and let $F$ be an arbitrary finite set. Then
\[Q_{U(F,n)} \geq \nicefrac54.\]
\end{thmn}

\begin{proof}
As enlarging $F$ only makes the statement harder to
prove, we can assume
that $F = \{3,4,\ldots,\ell\}$ for some $\ell \geq 11$. Let $s = \ell!$, fix a  ${t > 101}$, and let ${y = s \cdot t}$.
\begin{lem} \label{lemma:teilerfremd} 
	In the above setting, 
	\[\gcd(y+1, \br{10y} -1) = \gcd(y-1, \br{10y} -1) = \gcd(y+1, \br{10y} + 1) = 1.\] 
\end{lem} 
\begin{proof} 
Suppose that a prime~$p$ divides~${y+1}$. Then~${y \equiv -1 \pmod{p}}$ and therefore ${\br{10y} - 1\equiv -11 \pmod{p}}$. Then for ${p \ne 11}$ we clearly have ${p \nmid \gcd(y+1, \br{10y} -1)}$. On the other hand, since ${\ell \geq 11}$, we have that ${y = \ell! \cdot t \equiv 0 \pmod{11}}$, and thus $11$~isn't a divisor of ${y+1}$ either.

Analogously, if $p$ divides ${y-1}$ then ${y \equiv 1 \pmod{p}}$ and thus ${\br{10y} - 1 \equiv 9 \pmod{p}}$. Since ${p = 3}$ is a divisor of ${s = \ell!}$, we can conclude that ${\gcd(y-1, \br{10y} - 1) = 1}$.

Finally, if $p$ divides ${y+1}$ then ${10 y + 1 \equiv -9 \pmod{p}}$. Again, ${p = 3}$ is excluded by the choices made above, and thus ${\gcd(y+1, \br{10y} + 1) = 1}$.\phantom{\qedhere}\hfill$\Diamond$
\end{proof} 

\noindent
Note that there are infinitely many positive
integers $h_1$ such that 
\[{(y+1)^{h_1} \equiv 1 \mod{\br{10y} -1}}\]
as it suffices to choose $h_1$ as any multiple of the order of the coset of ${y+1}$ in the multiplicative group of the ring of residue classes modulo ${\br{10y} - 1}$. Analogously, there exist infinitely many integers~$h_2$ such that 
\[{(y+1)^{h_2} \equiv 1 \mod{\br{10y} +1}}.\]
Fixing some such $h_1$ and $h_2$, respectively, and letting $h$ be any integer greater than or equal to $\max(h_1, h_2)$, we have both
\[(y+1)^{h!} \equiv 1 \mod{\br{10y} -1} \quad \text{and} \quad (y+1)^{h!} \equiv 1 \mod{\br{10y} +1}.\]
Later, we will let $h$ go to infinity, but for the
moment we give an analysis that holds true independently of the exact
value of $h$ as long as it is sufficiently large.

\medskip
\noindent
So let $x=(y+1)^{h!}$. First note that since $y$ is even, $x$ is odd by definition. Secondly, it is clear that ${\gcd(x,y)=1}$ and in particular that
\begin{equation}\label{ex_dagger}
	\text{there is no } m \in {F \cup \{2\}} \text{ that divides } x.
\end{equation}
\begin{itemize} 
	\item We choose the first four entries of the $n$-tuple $(a_1,\ldots,a_n)$ as
\[
 a_1 = (x+y)^5,\quad
a_2 = -(x-y)^5,\]
\[a_3 = -(10 y - 1) \cdot x^4,\quad
 a_4 = -(x^2+\br{10y}^3)^2\!.
\]
\end{itemize} 
Of course we haven't fixed $h$ yet, so that the exact value of $x$ is
undetermined, and the same is consequently true for ${a_1,a_2,a_3, a_4}$. However,
we can already observe that 
\begin{equation}\label{ex_ddagger}
	a_1+a_2+a_3+a_4 = 2y^5 - 100 y^6
\end{equation}
and therefore that $a_1+a_2+a_3+a_4$ is independent of $x$.
We continue with the definition of $a_7,a_8,\ldots,a_n$ in a way that
does not depend on $x$, either:
\begin{itemize}
\item Let $a_7,a_8,\ldots,a_n$ be \br{negative} odd prime numbers
such that ${|a_7| > 200 y^6}$ and such that  ${|a_{k+1}| > 2 \cdot |a_{k}|}$ for ${k=7,8,\ldots,n-1}$. Then, using~\eqref{ex_ddagger},
$$|a_7| > 2 \cdot |a_1 + a_2 + a_3 + a_4|.$$ 
\end{itemize}
Finally, we need to fix the remaining two elements $a_5$ and $a_6$; by the
preceding choices and arguments the following definition is again
independent of $x$:
\begin{itemize}
\item Let \br{${u = a_1+a_2+a_3+a_4+(\sum_{k=7}^{n} a_k)}$}
and let \br{$m=-4u$}. By the previous choices, it is easy to see that $u$ must be a \br{negative} number. So it is possible to apply
Lemma~\ref{pr:factoravoidance} to $u$ and $m$
and let $a_5$ and $a_6$ be the numbers \br{$-v$ and $-w$} with $u=v+w$ as provided by that lemma.
\end{itemize}

\medskip
\noindent
We will show in a moment that, for every large enough $h$, the conditions in Definition~\ref{def:U} are met by $(a_1,\ldots,a_n)$. We
claim that this then implies that ${Q(F,n) \geq \nicefrac54}$; to see that, note that $\rad(a_1 \cdot \ldots \cdot a_n)$
will be a divisor of
\[
(x+y) \cdot (x-y) \cdot (\br{10y}-1)\cdot (y+1) \cdot (x^2+\br{10y}^3)
 \cdot a_5 \cdot \ldots \cdot a_n.\]
Letting $h$ go to infinity does not affect $a_5, a_6, \ldots, a_n$ at all.
Inside $a_1, a_2, a_3, a_4$, only~$x$~grows with~$h$ while all other terms remain constant. Thus, $\rad(a_1 \cdot \ldots
\cdot a_n)$ is bounded from above by a polynomial in $x$ of degree
at~most~$4$, while due
to the choice of~$a_1$ we have that $\max\{|a_1|,\ldots,|a_n|\}$ is bounded
from below by a
polynomial in~$x$ of degree~$5$. Therefore ${Q_{U(F,n)} \geq \nicefrac54}$.

\medskip
\noindent
It remains to show that, for all $h$ large enough, the four
conditions in Definition~\ref{def:U} are met by
$(a_1,\ldots,a_n)$. That condition~(ii) holds is immediate by the choice of~$a_5$ and~$a_6$. 

\br{If $p$ is an arbitrary prime factor of~$y$ then, since ${x \equiv 1 \mod{p}}$ by definition, it follows that each of $a_1$, $a_2$, $a_3$ and $a_4$ is congruent to~$\pm1$ modulo~$p$.} It follows that none of $a_1$, $a_2$, $a_3$, $a_4$ are divisible by any element of~$F$; and since the same is true for each of $a_5,a_6,\ldots,a_n$ by construction, condition~(iv) is satisfied.

Next, we establish condition (i) in several intermediate steps:
\begin{itemize}
\item {\em $a_1$ and $a_2$ are coprime:}
Note that any common prime divisor of $a_1$ and $a_2$ must also 
be a factor of $2y$, as it
must divide ${x+y}$ and ${x-y}$ and thus their difference.
Note that $y$ is even by construction, so that $y$ has the same prime
divisors as $2y$. Thus, any common prime divisor of $a_1$ and $a_2$
must also divide $y$ and, consequently, $x$.
But we already know that ${\gcd(x,y)=1}$.

\item {\em $a_3$ is coprime with both $a_1$ and $a_2$:} The factor $x$
of $a_3$ is coprime with $x+y$ and $x-y$, as $x$ is coprime with $y$.
Furthermore, 
\[ x=(y+1)^{h!} \equiv 1 \pmod{10 y -1}\]
by the choice of $h$, and thus 
\[\begin{array}{r@{\;}c@{\;}r@{}c}
	x+y & \equiv & 1+y  & \pmod{10  y -1}, \\
	x-y & \equiv & 1-y & \pmod{10 y -1}.
\end{array}
\] 
By Lemma~\ref{lemma:teilerfremd}, ${10  y -1}$ is coprime with both ${1+y}$ and ${1-y}$. 
Therefore $a_3$ is coprime with $a_1$ and $a_2$.

\item {\em $a_3$ and $a_4$ are coprime:}
We establish this by showing that 
$a_4$ is coprime with both factors of~$a_3$.
First, to determine $\gcd(10 y -1,a_4)$, note that 
$x  \equiv 1 \mod{10  y-1}$ and that 
\begin{align} \label{eq:100} 100 y^2 - 1 = (10 y -1) \cdot (10 y+1) \equiv 0 \mod{10  y -1}.\end{align} 
This implies 
$y^2+1 \equiv 101 y^2 \mod{10 y-1}$ and thus 
\[
\begin{array}{r@{\;}c@{\;}l@{}l}
	  x^2 + \br{10y}^3 & \equiv & 1 + \br{10y}^3               & \pmod{10 y-1}    \\
	              & =      & \multicolumn{2}{@{}l}{(\br{10y}-1) \cdot y^2+y^2+1 }                  \\
	              & \equiv & 101 y^2                 & \pmod{10 y-1}. 
\end{array}
\]
As $101$ is prime and ${10 y-1>101}$, they have no common factor.
Moreover, in view of~\eqref{eq:100}, any common factor of $y^2$ with $10 y-1$ would also have to be a
factor of $1$; as a result, $\gcd(10 y -1,a_4)=1$.

Secondly, we must determine 
\[ \gcd(x, a_4)=\gcd(x,x^2 + 10 y^3) = \gcd(x, \br{10y}^3).\] 
But by~\eqref{ex_dagger}, no divisor of $y$ nor any
element of ${F \cup \{2\}}$ divides $x$. Therefore, ${\gcd(x, a_4)=1}$. 

\item {\em $a_4$ is coprime with both $a_1$ and $a_2$:}
Clearly, $a_1 \cdot a_2$ is a power of 
\[ (x+y)(x-y)=x^2-y^2 \] while $a_4$
is a (negated) power of $x^2+10 y^3$. Any common prime factor~$p$ of~$a_4$ with either $a_1$ or $a_2$ would therefore have to be a factor
of the difference between these two expressions, that is, of 
${10 y^3+y^2 = y^2 \cdot (\br{10y}+1)}$. Such a~$p$ divides one of ${x+y}$ or
${x-y}$; thus, it cannot be a factor of $y$, because otherwise it would
divide $x$, contradicting the coprimeness of $x$ and~$y$.
Thus, such a $p$ would have to be a prime factor of ${\br{10y}+1}$.

Recall that $x$ was chosen such that ${x \equiv 1 \mod{\br{10y}+1}}$, 
thus we would have~${x \equiv 1 \mod{p}}$.
Since $p$ divides one of ${x+y}$ or ${x-y}$, it would also be a prime factor of
either 
\[ (\br{10y}+1) - 10\cdot (x+y) = -10x + 1 \equiv -9 \mod{p}\] or of
\[ (\br{10y}+1) + 10\cdot (x-y)= 10x + 1 \equiv 11  \mod{p}.\] 
This could only be true if $p\in\{3,11\}$, which is impossible since
both~$3$ and~$11$ divide $y$ and thus cannot divide $\br{10y}+1$. In
conclusion, $a_4$ is coprime with both $a_1$ and $a_2$.

\item {\em Each of $a_1,a_2,a_3$ is coprime with each of $a_5,a_6,\ldots,a_n$:}
By construction, $a_5$ and $a_6$ do not depend on $x$. For sufficiently large $h$ any prime factor of $a_5$ and $a_6$ is at most $h$.
Observe that for any prime $p$ with ${p \leq h}$ 
it holds that
$p-1$ divides $h!$, and thus, by Fermat's Little Theorem, 
\[{x=(y+1)^{h!} \equiv 1 \mod p}.\]
This holds in particular for any prime~$p$ dividing $a_5$ or $a_6$ and, if $h$ is large enough, for all  primes~${p\in\{a_7,a_8,\ldots,a_n\}}$; thus any such $p$ is coprime with~$x$. 
Moreover, we have 
\begin{eqnarray*} x + y & \equiv & 1 + y \mod{p}, \\ 
	x -y & \equiv & 1 - y \mod{p} \end{eqnarray*} 
for these primes $p$. Since for  such a~$p$ we also have ${p > 200y^6 > y \pm 1}$ by construction, it follows that $p$~is coprime with $x+y$ and $x-y$ as well. As we trivially have $\br{p \nmid (10 y-1)}$, we can conclude that $p$ 
does not divide any of~${a_1,a_2,a_3}$.

\item {\em $a_4$ is coprime with each of $a_5, a_6,\ldots, a_n$:} 
For the same reasons as in the previous item, we only need to consider potential
prime factors $p$ between ${200 y^6}$ and~$h$. For such $p$, we again have that 
$x \equiv 1 \mod{p}$. Then
\[x^2+\br{10y}^3 \equiv 1+\br{10y}^3 \not\equiv 0 \mod p,\]
which implies that $p$ does not divide $a_4$. 

\item {\em $a_5,a_6,\ldots,a_n$ are pairwise coprime:} 
First, being pairwise distinct primes, the numbers~$a_7,a_8\ldots,a_n$ are trivially pairwise coprime. Secondly, recall how $a_5$ and $a_6$ were defined using Lemma~\ref{pr:factoravoidance}
in such a way as to ensure that $a_5$ and $a_6$ are coprime with each other. Finally, Lemma~\ref{pr:factoravoidance} also guarantees that no primes
less than $m$ divide $a_5$ or $a_6$; and as \br{${m = -4u}$} is larger  than any of~$|a_i|$ for~${7 \leq i \leq n}$ \br{by the choice of $u$}, we have in particular that both of $a_5$ and $a_6$ are coprime with each of~${a_7,a_8,\ldots,a_n}$. 
\end{itemize}

\noindent
It remains to establish subsum condition~(iii) for
$(a_1,\ldots,a_n)$. Assume that we have fixed~$b_1,\ldots,b_n \in \{-1,0,+1\}$ such that
$\sum_{k=1}^n b_k \cdot a_k = 0$. We proceed via a series of claims:

\begin{itemize}
\item {\em It must hold that $b_1=b_2=b_3=b_4$:} 
Recall that $a_5,a_6\ldots,a_n$ do not depend on~$x$. Since $x = (y+1)^{h!}$,
this implies for $h$ large enough that
${x > |a_5|+|a_6|+\ldots+|a_n|}$. Thus, if for some choice of $(b_1,b_2,b_3,b_4)$ we have that $|\sum_{k=1}^4 b_k \cdot a_k|  > x$, then {\em no} choice of $(b_5,b_6,\ldots,b_n)$ can lead to ${\sum_{k=1}^n b_k \cdot a_k = 0}$. We will argue that this must be the case unless we have $b_1=b_2=b_3=b_4$.

So let us inspect all possible choices of $(b_1,b_2,b_3,b_4)$. We first exclude some trivial cases: First, if only one of $b_1,b_2,b_3,b_4$ is non-zero, then clearly $|\sum_{k=1}^4 b_k \cdot a_k| > x$. Secondly, we can omit choices of 
$(b_1,b_2,b_3,b_4)$ where the summands ${b_k \cdot a_k}$ are all positive or all negative. Finally, to further reduce the numbers of cases to consider, we assume w.l.o.g.\ that $b_i=1$ when $1 \leq i \leq 4$ is smallest with $b_i \neq 0$; the case $b_i=-1$ is symmetric. Then the following cases not satisfying $b_1=b_2=b_3=b_4$ remain:
\begin{itemize}
\item $a_1+a_2+ \br{b_3 \cdot a_3} +b_4 \cdot a_4$ where $(b_3,b_4)\neq(1,1)$,
\item $a_1-a_2+ \br{b_3 \cdot a_3} +b_4 \cdot a_4$ where ${b_3, b_4 \in \{-1,0,+1\}}$, 
\item $a_1+b_3 \cdot a_3 + b_4 \cdot a_4$ where ${b_3, b_4 \in \{-1,0,+1\}}$,
\item $a_2+b_3 \cdot a_3+b_4 \cdot a_4$ where ${b_3, b_4 \in \{-1,0,+1\}}$,
\item $a_3+b_4 \cdot a_4$ where ${b_4 \in \{-1,0,+1\}}$,
\end{itemize}
and the absolute values of all of these expressions are easily seen to be lower-bounded by~$x$.
\end{itemize}

\noindent
  With the preceding claim established, we can from now on treat 
$a_1+a_2+a_3+a_4$  as a \textit{single} number that can either be included in a subsum with positive or negative sign, or not.

\begin{itemize}
\item {\em It must hold that $b_5=b_6$:} Note that by the choice of $a_5$ and $a_6$ and by the properties ensured by Lemma~\ref{pr:factoravoidance} we have that $a_5>0$ and $a_6<0$ and that 
\[
\begin{array}{r@{\;}c@{\;}l}
	\br{|a_5|, a_6} &>&\Big|\phantom{|}a_1+a_2+a_3+a_4\phantom{|} + \sum_{k=7}^{n} \phantom{|}a_k\phantom{|}\Big|\\[0.8em]
	&=& \phantom{\Big|}	|a_1+a_2+a_3+a_4|+\sum_{k=7}^{n} |a_k|;
\end{array}
\]
here the equality uses the fact that $a_1+a_2+a_3+a_4$ is negative by~\eqref{ex_ddagger} while $a_7,a_8\ldots,a_n$ are negative by choice. As a consequence, in any subsum equaling zero, $a_5$ and $a_6$ must either not occur at all or in such a way that they partly cancel each other out additively. This is only possible when~${b_5=b_6}$.
\end{itemize}

\noindent
  Again, from now on we treat $a_5+a_6$ as a single number
that may be part of a subsum or not. To complete the proof we distinguish all three possible cases concerning the value of ${b_5=b_6}$.

\begin{itemize}
\item {\em If $b_5=b_6=0$, then the subsum is empty:}
This is because in the sequence
\[|a_1+a_2+a_3+a_4|,|a_7|,|a_8|,\ldots,|a_n|\]
each entry is more than $2$ times larger than the previous one; thus the only way of obtaining a zero subsum in this case is when ${b_k=0}$ for all ${1\leq k \leq n}$.

\item {\em If $b_5=b_6=1$, then $b_k=1$ for all ${1\leq k \leq n}$:}
Assume that for some choice of $(b_k)_{1\leq k \leq n}$ with ${b_5=b_6=1}$ we have ${\sum_{k=1}^n b_k \cdot a_k = 0}$. Since we also have ${\sum_{k=1}^n a_k = 0}$ it follows that
\[\begin{array}{c@{\;}l}
	&\sum_{k=1}^n a_k - \sum_{k=1}^n b_k \cdot a_k \\[0.5em]
	=& (1-b_1) \cdot (a_1+a_2+a_3+a_4)+\sum_{k=7}^n (1-b_k) \cdot a_k\\[0.5em]
	=& 0,
\end{array}
\]
where $1-b_k \in \{0,1,2\}$ for $k \in \{1,7,8,\ldots, n\}$. For the same reason as in the previous item, the only choice of $(1-b_k)_{k \in \{1,7,8,\ldots, n\}}$ that makes this equality true is $1-b_k =0 $ (thus $b_k =1 $) for all $k \in \{1,7,8,\ldots, n\}$.

\item {\em If $b_5=b_6=-1$, then $b_k=-1$ for all $1\leq k \leq n$,} by a symmetric argument.
\end{itemize}

\noindent
Thus condition (iii) holds, completing the proof.
\end{proof}

\section{\rr{Closing} remarks}

\noindent
In the preceding sections we established new lower bounds for strong variants of the $n$-conjecture. In that context, we always exclusively considered $n$-tuples of pairwise coprime integers. To conclude the article, we make some closing remarks about instances that are {\em not} necessarily pairwise coprime.

\begin{rem}\label{rem:halb}
	If we were to allow common factors in $n$-tuples, we could for example consider the set of quadruples of the form
	\[{((2^h+1)^3,-2^{3h},-3 \cdot 2^h \cdot(2^h+1),-1)}\]
	for ${h \in \mathbb N}$. Note that we still have ${\gcd(a_1, a_2, a_3, a_4) = 1}$, but that arbitrarily large common divisors occur between pairs of these numbers; for instance, both~$a_2$ and~$a_3$ are divisible by 
	$2^h$.  It is not too difficult to see that these quadruples belong to the set $A(4)$ from Conjecture~\ref{conj:bb}. The limit superior of the qualities of these quadruples is~$3$; that is, under these relaxed conditions, it is possible to achieve considerably larger qualities than in the preceeding sections. 
	
	This is in accordance with Conjecture~\ref{conj:bb} and the previously known result of Browkin and Brzezi\'nski~\cite[Theorem~1]{BB94} that ${Q_{A(n)} \ge 2n -5}$ for every ${n \ge 3}$. 	The proof of this fact starts from a geometric sum equation 
	\[ \sum_{i=0}^{k-3} y^i = \frac{y^{k-2} - 1}{y-1}.\] 
	Multiplying both sides of the equation by ${x := y-1}$ we obtain 
	\[ y^{k-2} - xy^{k-3} - xy^{k-4} - \ldots - x - 1 = 0. \] 
	It is easy to see that conditions~(i) and (iii) from Conjecture~\ref{conj:bb} are satisfied. For a clever choice of $k$ and ${x = y-1}$, Browkin and Brzezi\'nski were able to obtain a sequence of $n$-tuples summing to zero such that each single \br{$n$-tuple} satisfies the subsum condition and such that the sequence of corresponding qualities has an accumulation point ${\geq 2n-5}$. 
\end{rem}

\noindent
The previous comments concern the case where we allow unbounded common divisors between the elements of the solution $n$-tuples. This can be thought of as the opposite extreme of the situation studied in the main parts of this article where we only considered $n$-tuples whose entries were required to be pairwise coprime. In between these two extremes, we could also study a case where finitely many, that is bounded, common divisors are permitted. We conclude the article by giving an example of an intermediate result that can be obtained for this setting.

\begin{lem}
There is a finite set $E$ such that there exists a sequence $(a^{(h)})$ of quintuples $(a_1, a_2,a_3,a_4,a_5)$ of integers such that 
\begin{enumerate}[(i)]
	\item ${\gcd(a_i,a_j) \in E}$ for any ${1 \le i < j \le 5}$; 
	\item $a_1+\ldots+a_5=0$;
	\item there are no $b_1,\ldots,b_5 \in \{-1,0,1\}$ and $i,j$ with
	$1\leq i,j\leq 5$ such that $b_i=0$ and $b_j=1$ and
	$\sum_{k=1}^5 b_k \cdot a_k = 0$; 
	\item ${\gcd(a_1, \ldots, a_5) = 1}$; and 
	\item ${\limsup_{h \rightarrow \infty} q(a^{(h)}) \geq \nicefrac95}$.
\end{enumerate}
\end{lem}

\begin{proof}
Let $x$ be ${\ell^h-1}$ for some $h\in \mathbb{N}$ and some fixed odd prime number $\ell$. Fix
\[a_1 = 189(x+1)^9, \quad
a_2 = -189(x-1)^9,\quad
a_3 = -42(3x^2+7)^4,
\]\[
a_4 = 16(63x^2+79)^2,\quad
a_5 = 608.
\]
The greatest common
divisor of $a_1$ and $a_2$ is $189$. 

We claim that $\gcd({a_1 \cdot a_2},a_3)$ divides $1890$. 
To see this, first note that on the one hand, the least common multiple of $189$ and $42$ is $378$. On the other hand, if a prime $p$ divides both ${x^2-1}$ and ${3x^2 + 7}$, then~${x^2 \equiv 1 \!\pmod{p}}$ and therefore ${3x^2 + 7 \equiv 10\!\pmod{p}}$; thus, since $p$~divides ${3x^2 + 7}$ by assumption, we may conclude that $p$ is a factor of~$10$. Note that 1890 is the least common multiple of $378$ and $10$. 

In an analogous way, we can argue that  $\gcd({a_1 \cdot a_2},a_4)$ is
a factor of $214704$ and that $\gcd(a_3, a_4)$ is a factor of $5712$.
The greatest common divisor of $a_5$ and any $a_i$ with ${i \ne 5}$ is a factor of $608$. Then, letting 
$$E = \{r\colon r \textup{ divides one of }
608, 1890, 5712, 214704\},$$ 
we have that all common divisors of the entries of $(a_1,\ldots,a_5)$ are contained in $E$. We also note that ${\gcd(a_1, \ldots, a_5)}$ divides ${\gcd(\gcd(a_1, a_2), a_5) = 1}$. 

An easy calculation shows that (ii) is satisfied. To see that condition~(iii) holds, argue as in the proof of Theorem~\ref{satz_polynom_begr}.

By definition, $x+1$ is a power of~$\ell$. Thus, ${\rad(a_1 \cdot a_2 \cdot a_3 \cdot a_4 \cdot a_5)}$ is a factor of 
\[189 \cdot 42 \cdot 16 \cdot 608 \cdot \ell \cdot
(x-1) \cdot (3x^2+7) \cdot (63x^2+79).\]
As this is a polynomial of  degree $5$, whereas $a_1$ is a polynomial of degree $9$, we conclude that ${\limsup_{h \rightarrow \infty} q(a^{(h)}) \geq \nicefrac95}$.
\end{proof}

\subsubsection*{Acknowledgements.} The authors would like to
thank Benne de Weger for helpful correspondence and for making his unpublished
notes~\cite{Weg20} available to them. \rr{Moreover we would like to thank the anonymous referee for the fast and very thorough proof-reading of our manuscript, as well as for suggesting important improvements.}

\end{document}